\documentclass[11pt,a4paper]{amsart}

\usepackage{amsmath, amsfonts, amsthm, latexsym, amscd}

\usepackage{amsfonts}
\usepackage{amssymb}

\usepackage{enumerate}

\usepackage[dvipsnames]{xcolor}

\usepackage{hyperref}
\hypersetup{
colorlinks=true,
linkcolor=blue,
citecolor=cyan,
urlcolor=NavyBlue,
}

\usepackage[norefs,nocites]{refcheck}

\allowdisplaybreaks

\usepackage{geometry}
\newtheorem{theorem}{Theorem}[section]
\newtheorem{proposition}[theorem]{Proposition}
\newtheorem{corollary}[theorem]{Corollary}
\newtheorem{example}[theorem]{Example}

\newtheorem{remark}[theorem]{Remark}

\newtheorem{lemma}[theorem]{Lemma}
\newtheorem{final remark}[theorem]{Final Remark}
\newtheorem{definition}[theorem]{Definition}

\begin{document}

\title{Geometric spaceability in sequence classes and operator ideals}

\author[N.G.~Albuquerque]{Nacib G.~Albuquerque$^\dag$} 
\address[N.G.~Albuquerque]{Departamento de Matem\'{a}tica \newline\indent
Universidade Federal da Para\'{i}ba \newline\indent
Jo\~ao Pessoa - PB \newline\indent
58.051-900 (Brazil)}
\email{\href{mailto:ngalbuquerque@mat.ufpb.br}{ngalbuquerque@mat.ufpb.br}}
\thanks{$^\dag$Corresponding author. E-mail: \href{mailto:ngalbuquerque@mat.ufpb.br}{\tt ngalbuquerque@mat.ufpb.br}}

\author[J.R.~Campos]{Jamilson R. Campos}
\address[J.R.~Campos]{Departamento de Matem\'{a}tica \newline\indent
Universidade Federal da Para\'{i}ba \newline\indent
Jo\~ao Pessoa - PB \newline\indent
58.051-900 (Brazil)}
\email{\href{mailto:jamilsonrc@gmail.com}{jrc@academico.ufpb.br, jamilsonrc@gmail.com}}

\author[L.F.P.~Sousa]{Luiz Felipe P. Sousa}
\address[L.F.P.~Sousa]{Departamento de Matem\'{a}tica \newline\indent
Universidade Federal da Para\'{i}ba \newline\indent
Jo\~ao Pessoa - PB \newline\indent
58.051-900 (Brazil)}
\email{\href{mailto:lfpinhosousa@gmail.com}{lfpinhosousa@gmail.com}}



\begin{abstract} 
This paper investigates advanced notions of lineability and spaceability within the frameworks of sequence spaces and operator ideals. We propose the notion of \emph{Standard Sequence Classes} to provide an environment that unifies numerous classical sequence spaces while preserving their fundamental behavior. Utilizing this framework, we establish general $(\alpha, \mathfrak{c})$-spaceability results for complements of unions of (quasi-)Banach sequence spaces. These results extend the existing literature by addressing the geometrically more demanding case where $\alpha > 1$ and by encompassing the non-locally convex (quasi-)Banach setting. Furthermore, we provide criteria for the pointwise $\mathfrak{c}$-spaceability of differences between general operator ideals with values in standard sequence spaces. Our results recover and improve several known findings in the context of vector-valued sequences.
\end{abstract}

\maketitle

\section{Introduction and Background}

The property known as \emph{lineability}, introduced by Aron, Gurariy, and Seoane-Sepúlveda in the seminal work \cite{Aron-Gurariy-Seoane}, refers to the search for linear structures within sets that, at first glance, appear to lack any linear character. More precisely, given a cardinal number $\mu$ and a linear space $V$, a subset $A \subset V$ is said to be $\mu$-\emph{lineable} if $A \cup \{0\}$ contains a $\mu$-dimensional linear subspace of $V$. When $V$ is a topological vector space and the subspace is closed, $A$ is called $\mu$\emph{-spaceable}. Since its formalization, lineability theory has been extensively developed in several directions. For a more in-depth study, we recommend \cite{bernal-jfa,BerPS,Gurariy-Quarta} and the book \cite{Lineability-book}.

In this work, we emphasize the setting of sequence spaces and operator ideals. As the theory evolved, more refined notions of lineability were introduced. Among them are \emph{pointwise lineability}, introduced in \cite{Pellegrino-Raposo}, and \emph{$(\alpha,\beta)$-lineability}, a notion with geometric flavor introduced in \cite{Favaro-Pellegrino-Tomaz}. Given cardinals $\alpha \leq \beta$ and a (topological) vector space $V$, a subset $A \subset V$ is said to be:

\vskip 2mm
\noindent$\bullet$ \emph{$(\alpha,\beta)$-lineable (spaceable)} if it is $\alpha$-lineable and, for every 
$\alpha$-dimensional subspace $U_\alpha \subset A \cup \{0\}$, there exists a $\beta$-dimensional (closed) subspace $U_\beta$ of $V$ such that
\(
U_\alpha \subset U_\beta \subset A \cup \{0\}
\) (see \cite[Definition 1.1]{Favaro-Pellegrino-Tomaz}).

\vskip 2mm
\noindent$\bullet$ \emph{pointwise $\alpha$-lineable (spaceable)} if, for every $x \in A$, there exists an $\alpha$-dimensional (closed) 
subspace $W_x$ with
\(
x \in W_x \subset A \cup \{0\}
\) (see \cite[Definition 2.1]{Pellegrino-Raposo}).
\vskip 2mm

The extensive literature on lineability and spaceability in sequence spaces has resolved many questions within the classical framework of the theory (see \cite{Araujo-Barbosa-Raposo-Ribeiro, Barroso-Botelho-Favaro-Pellegrino, Botelho-Cariello-Favaro-Pellegrino, Botelho-Diniz-Favaro-Pellegrino, Botelho-Favaro} and \cite{Nogueira-Pellegrino}). Nevertheless, several problems remain open. In \cite{Pellegrino-Raposo}, Pellegrino and Raposo established results concerning the pointwise $\mathfrak{c}$-spaceability of sets that generalize complements of unions of Banach and quasi-Banach sequence spaces. Since pointwise $\mathfrak{c}$-spaceability implies $(1,\mathfrak{c})$-spaceability, a natural question arises regarding the behavior of $(\alpha,\mathfrak{c})$-spaceability for $\alpha > 1$. Recently, Ara\'ujo et al.\ \cite{Araujo-Barbosa-Raposo-Ribeiro} provided general criteria for $(\alpha,\mathfrak{c})$-spaceability in complements of unions of Banach spaces. When applied to sequence spaces, these criteria partially address the aforementioned question.

One of the principal aims of the present work is to establish results concerning the geometric spaceability notion within the $(\alpha,\beta)$-setting for general sets similar in nature to those studied in \cite{Pellegrino-Raposo}. This endeavor provides answers for the challenging cases where $\alpha > 1$ and, crucially, addresses the setting of quasi-Banach spaces, which is not covered by the general criteria established in \cite{Araujo-Barbosa-Raposo-Ribeiro}. To achieve this, we introduce the concept of \emph{Standard Sequence Classes}, inspired by the sequence space notions in \cite{Albuquerque-Coleta} and the abstract framework of Sequence Classes developed by Botelho and Campos in \cite{Botelho-Campos}.  For a comprehensive overview of recent advances and applications within the environment of sequence classes, we refer the reader to \cite{Botelho-Campos-Nascimento, Botelho-Santiago-Repres, Campos-Nascimento-Sousa} and the references therein.

Standard sequence classes serve as a habitat for sequence classes that possess natural and well-behaved properties, which will be detailed later.  As an illustrative consequence of one of our main results (Theorem \ref{th-standardseq} and Corollary \ref{cor1}), consider a Banach space $E$, and $E$-valued quasi-Banach standard sequence classes $X$ and $Y_\lambda$ ($\lambda \in \Lambda$). If the set
$$
X(E) \setminus\bigcup_{\lambda \in \Lambda}Y_\lambda(E)
$$
is non-empty, then it is $(\alpha, \mathfrak{c})$-spaceable if and only if $\alpha < \aleph_0$.

In the context of operator ideals, several contributions have advanced the study of lineability and spaceability for various classes of operators. A seminal result in this area is due to Kitson and Timoney \cite{Kitson-Timoney}, who provided spaceability results for the difference of certain operator ideals on Banach spaces; however, the non-locally convex setting was not covered. Relevant contributions regarding spaceability results for the difference of operator classes can also be found in \cite{Alves-Turco,Araujo-Pellegrino,Pellegrino1,Favaro-Pellegrino-Rueda} and \cite{Puglisi}. In \cite{Sanchez}  Hern\'andez et al. investigated spaceability for the general difference $\mathcal{I}_1 \setminus \mathcal{I}_2$, where $\mathcal{I}_i$ ($i=1,2$) denotes a Banach operator ideal on Banach spaces satisfying specific properties. More recently, \cite{Albuquerque-Coleta} established results on pointwise $\mathfrak{c}$-spaceability in the setting of multilinear operators taking values in standard sequence spaces, focusing on particular classes of operators, such as bounded, multiple, and absolutely summing classes.

Motivated by these contexts, another main result established in this paper (Theorem \ref{th-opideals}) addresses the difference of a general (quasi-)Banach operator ideal with values in standard sequence classes. This result improves and generalizes several of the aforementioned previous findings. As a particular application, we provide the pointwise $\alpha$-spaceability of
\[
\mathcal{I}(E;\mathcal{S}) \setminus \Pi_{p,q}(E;\mathcal{S}),
\]
where $\mathcal{I}$ is a (quasi-)Banach operator ideal, $\mathcal{S}$ is a standard sequence space, $\Pi$ denotes the class of (absolutely) summing operators, and $p,q$ satisfy mild conditions.

The paper is organized as follows. In Section~\ref{SEC-SSC} we introduce the concept of standard sequence classes, present the main examples and basic properties, and define the sets to be studied in the subsequent sections. In Section~\ref{Sec-LIN-SSS} we prove the main $(\alpha,\beta)$-spaceability result for sets that generalize complements of unions of sequence spaces. Section~\ref{Sec-PL-OI} is devoted to the study of pointwise lineability in the context of general difference between operator ideals with values in standard sequence spaces. Finally, in Section~\ref{SEC-APLIC} we collect several applications of the main results established in the previous sections.

\subsection{Notations and terminology}

For Banach spaces $E$ and $F$ over $\mathbb{K} = \mathbb{R}$ or $\mathbb{C}$, $E^*$ denotes the topological dual of $E$, $B_E$ denotes the closed unit ball of $E$ and ${\mathcal L}(E;F)$ denotes the Banach space of bounded linear operators from $E$ to $F$ with the usual operator norm. The symbol $E \stackrel{C}{\hookrightarrow} F$ means that $E$ is a linear subspace of $F$ and $\|\cdot\|_F \leq C\|\cdot\|_E$ on $E$. We write $E \stackrel{1}{=}F $ if $E = F$ isometrically.

The spaces of eventually null and bounded $E$-valued sequences are denoted by $c_{00}(E)$ and $\ell_\infty(E)$, respectively. For $j \in \mathbb{N}$ and  $x \in E$, we write $x \cdot e_j$ for the sequence $(0,\dots, 0,x,0, 0,\dots )$, where $x$ is placed at the $j$-th coordinate. The symbol $(x_{j})_{j=1}^{n}$, where $x_1, \ldots, x_n \in E$, stands for the sequence $(x_{1},x_{2},\ldots,x_{n},0,0,\ldots) \in E^\mathbb{N}$. If $x = (x_j)_{j=1}^\infty \in E^\mathbb{N}$ is a sequence, in addition to the standard notation $x_j$ for the $j$-th coordinate of $x$, we will write $x(j)$ when more precision is needed. We will also use the notation $x|_{\mathbb{N}'}$ to represent the subsequence $(x_j)_{j \in \mathbb{N}'}$, where $\mathbb{N}'$ is an infinite subset of $\mathbb{N}$. When we refer to a partition  $\mathbb{N} = \mathbb{N}_1 \cup \mathbb{N}_2$, we mean that $\mathbb{N}_1$ and $\mathbb{N}_2$ are infinite and disjoint subsets of $\mathbb{N}$. The symbols $\aleph_0$ and $\mathfrak{c}$ denote the cardinalities of $\mathbb{N}$ and $\mathbb{R}$, respectively. We shall work under the Continuum Hypothesis, that is, we assume that no cardinal number lies strictly between $\aleph_0$ and $\mathfrak{c}$.

\section{Standard sequence classes}\label{SEC-SSC}

In this section we introduce our central object of study and its properties. The inspiration used in the definition is closely related to the studies developed in \cite{Albuquerque-Coleta} and \cite{Botelho-Campos}, where the concepts of Standard Sequence Spaces and Sequence Classes are defined, and in \cite{Nogueira-Pellegrino} and \cite{Pellegrino-Raposo}, where are defined the concept of Strongly Invariant Sequence Spaces.

\begin{definition}\rm \label{standard-sc}
A \emph{standard sequence class} $X$ is a rule that assigns to each Banach space $E$ a (quasi-)Banach space $X(E)$ of $E$-valued sequence space (endowed with the usual coordinatewise algebraic operations) satisfying:
\begin{enumerate}[(i)]
\item There exists a constant $C>0$ such that
$ c_{00}(E) \subseteq X(E) \stackrel{C}{\hookrightarrow} \ell_\infty(E)$.

\item (\emph{Subsequence stability}) If $x=(x_j)_{j\in\mathbb{N}}\in X(E)$ and $(x_{n_k})_{k\in\mathbb{N}}$ is a subsequence of $x$, then $(x_{n_k})_{k\in\mathbb{N}}\in X(E)$ and
\[ \|(x_{n_k})_{k\in\mathbb{N}}\|_{X(E)}\leq \|x\|_{X(E)}.\]

\item (\emph{Spreading zeros stability}) If $x = (x_j)_{j\in\mathbb{N}} \in X(E)$ and $\mathbb{N}'=\{n_1<n_2<n_3<\cdots\}$ is an infinite subset of $\mathbb{N}$, then the $E$-valued sequence $y$ defined by
\[
y = \sum_{i \in \mathbb{N}} x_i \cdot e_{n_i}
\]
belongs to $X(E)$ and satisfies $ \|y\|_{X(E)} \leq \| x \|_{X(E)}$.
\end{enumerate}
\end{definition}

Condition i) ensures that convergence in \(X(E)\) implies coordinatewise convergence, or equivalently, that the projections \(\pi_m : X(E) \to E\), defined by \(\pi_m(x) = x_m\), are continuous for all \(m \in \mathbb{N}\). Conditions (ii) and (iii) imply that $X(E)$ is closed under taking subsequences and also it is stable in adding zeros in its sequences (the spreading zeros property). Moreover, by conditions ii) and iii), since \(x\) is a subsequence of \(y\), one has \(\|y\|_{X(E)} = \|x\|_{X(E)}\). Any spaces of the form \(X(E)\), where \(E\) is a Banach space and \(X\) a standard sequence class, will be called \emph{standard sequence spaces}.

\begin{example}\rm\label{ex-seqclass}
Next, we list several classical sequence spaces for which the correspondence \(E \mapsto X(E)\) defines (or not) a standard sequence class.

\vskip 2mm

\noindent a) For $p \in (0,\infty)$, the classical sequence spaces 
$\ell_p(E)$, $\ell_p^u(E)$, and $\ell_p^w(E)$ 
(see \cite[Section~8]{Defant-Floret}).

\vskip 2mm

\noindent b)  The space $c_0(E)$ of null convergent sequences.

\vskip 2mm

\noindent c) For $0 < p, q < \infty$, the Lorentz space $\ell_{p,q}(E)$ (see \cite{Matos}), given by
\[
\ell_{p,q}(E) :=  
\left\{ (x_j)_{j=1}^{\infty} \in c_0(E) : 
\|(x_n)_{n=1}^\infty\|_{p,q} =
\left( \sum_{n=1}^\infty
\left( n^{\frac{1}{p}-\frac{1}{q}} \|x_{\Phi(n)}\| \right)^q 
\right)^{\!\frac{1}{q}} < \infty \right\},
\]
where $\Phi: \mathbb{N} \to \mathbb{N}$ is an injective mapping such that $\|x_{\Phi(1)}\| \geq \|x_{\Phi(2)}\| \geq \dots \geq 0$, and $\Phi^{-1}(n) \neq \emptyset$ whenever $x_n \neq 0$.

\vskip 2mm

\noindent d) An \textit{Orlicz function} is a convex and non-decreasing function $M: [0,\infty) \to [0,\infty)$ such that $M(0) = 0$ and $\lim_{t \to \infty} M(t) = \infty$. If, in addition, there exists $t \neq 0$ such that $M(t) = 0$, we say that $M$ is a \textit{degenerate Orlicz function}. For a non-degenerate Orlicz function $M$, with $M(1)=1$, we may consider the Orlicz sequence space $\ell_M(E)$ (see \cite{Campos-Sousa} and \cite[Section~4.a]{Lindenstrauss}) and $\ell_M^w(E)$ (see \cite[Section 2]{Campos-Sousa}), given by
\[
\ell_M(E) :=
\biggl\lbrace (x_j)_{j=1}^\infty \in E^{\mathbb{N}} : 
\sum_{j=1}^\infty M\!\left( \frac{\|x_j\|}{\eta}\right) < \infty 
\text{ for some } \eta > 0
\biggr\rbrace.
\]
and
\[
\ell_M^w(E) :=
\biggl\lbrace (x_j)_{j=1}^\infty \in E^{\mathbb{N}} : 
(\varphi(x_j))^\infty_{j=1} \in \ell_M
\text{ for all } \varphi \in E'
\biggr\rbrace.
\]

\noindent e) For $1 \leq p < \infty$, the space $\ell_p\langle E\rangle$ of Cohen strongly $p$-summable sequences (see \cite{Cohen}), and the space $\ell_p^{\mathrm{mid}}(E)$ of mid $p$-summable sequences (see \cite{Botelho-Campos-Santos}), given by
\[
\ell_p\langle E \rangle :=  
\left\{ (x_j)_{j=1}^{\infty} \in E^{\mathbb{N}} : 
\|(x_j)_{j=1}^\infty\|_{\ell_p\langle E \rangle}
= \sup_{(\varphi_j)_{j=1}^\infty \in B_{\ell_{p^*}^w(E')}} 
\|(\varphi_j(x_j))_{j=1}^\infty\|_1 < \infty
\right\},
\]
and
\[
\ell_p^{\mathrm{mid}}(E) :=  
\left\{ (x_j)_{j=1}^{\infty} \in E^{\mathbb{N}} : 
\|(x_j)_{j=1}^{\infty}\|_{p,\mathrm{mid}} 
= \sup_{(\varphi_n)_{n=1}^\infty \in B_{\ell_p^w(E')}} 
\left(\sum_{n = 1}^\infty\sum_{j = 1}^\infty 
|\varphi_n(x_j)|^p \right)^{1/p} < \infty
\right\}.
\]

\noindent f) The spaces $c(E)$ (see \cite[p.~33]{Diestel}) and 
\[
G(E)=\{(x_n)_{n=1}^\infty \in \ell_\infty(E) : x_{2n-1}=x_{2n}\}
\]
are examples of sequence spaces that do not satisfy all the conditions in Definition \ref{standard-sc}. Indeed, for any $x \in E \setminus \{0\}$ and $\mathbb{N}'=\{2n : n \in \mathbb{N}\}$, we have $(x,x,\dots) \in c(E)$ but $(x,0,x,0,\dots) \notin c(E)$. Conversely, although $(x,0,0,\dots) \in c_{00}(E)$, this sequence does not belong to $G(E)$ and so $c_{00}(E) \not\subset G(E)$.
\end{example}

To avoid ambiguity, we shall occasionally denote a standard sequence class by \( X(\cdot) \). For example, we simply write \( \ell_p^w \) instead of \( \ell_p^w(\cdot) \). However, when referring to the sequence class \( \ell_p(\cdot) \), we keep the argument to distinguish it from the scalar sequence space \( \ell_p \).

We first prove a simple result for `nested' standard sequence classes, adapting the ideas from \cite[Lemma 3.2]{Pellegrino-Raposo}. Let $E$ be a Banach space and $\{X_\lambda\}_{\lambda \in \Lambda}$ a family of standard sequence classes. This family is said to be \emph{$E$-nested} if for every $\lambda_1, \lambda_2 \in \Lambda$, either $X_{\lambda_1}(E) \subseteq X_{\lambda_2}(E)$ or $X_{\lambda_2}(E) \subseteq X_{\lambda_1}(E)$. The family $\{X_\lambda\}_{\lambda \in \Lambda}$ is called \emph{nested} if it is $E$-nested for every Banach space $E$.

\begin{proposition} \label{lema-restricoes}
Let $\{X_\lambda\}_{\lambda \in \Lambda}$ be an $E$-nested family of standard sequence classes.
\begin{enumerate}[i)]
\item Given a partition $\mathbb{N} = \mathbb{N}' \cup \mathbb{N}''$ and $y \in E^{\mathbb{N}}$ such that $y \notin \bigcup_{\lambda \in \Lambda}X_\lambda(E)$, then 
\[
y|_{\mathbb{N}'} \notin \bigcup_{\lambda \in \Lambda}X_\lambda(E)
\quad \textrm{or} \quad
y|_{\mathbb{N}''} \notin \bigcup_{\lambda \in \Lambda}X_\lambda(E),
\]

\item Let $x^{(1)},\dots, x^{(n)} \in E^\mathbb{N}$ such that $x^{(1)},\dots, x^{(n)} \notin \bigcup_{\lambda \in \Lambda}X_\lambda(E)$. Then
\[
x^{(k)}|_{\mathbb{N}'} \notin \bigcup_{\lambda \in \Lambda}X_\lambda(E),\ \text{ for all } k=1,\ldots,n,
\]
for some partition  $\mathbb{N} = \mathbb{N}' \cup \mathbb{N}''$.
\end{enumerate}
\end{proposition}

\begin{proof}
Item i) is proved in \cite[Lemma 3.2]{Pellegrino-Raposo}. Briefly, the argument is as follows: if both $y|_{\mathbb{N}'}$ and $y|_{\mathbb{N}''}$ belong to $\cup_{\lambda \in \Lambda}X_\lambda(E)$, then the properties of standard sequence classes yield sequences $y_1, y_2 \in \cup_{\lambda \in \Lambda}X_\lambda(E)$ such that $y = y_1 + y_2$. This implies $y \in \cup_{\lambda \in \Lambda}X_\lambda(E)$, a contradiction.

\vskip 2mm

Now we prove ii): starting from $x^{(1)} \notin \bigcup_{\lambda \in \Lambda}X_\lambda(E)$, applying the first item, we can consider a partition $\mathbb{N} = \mathbb{N}_1 \cup \widetilde{\mathbb{N}}_1$ such that $x^{(1)}|_{\mathbb{N}_1} \notin \bigcup_{\lambda \in \Lambda}X_\lambda(E)$. Applying item i) once again, we have 
$$
x^{(2)}|_{\mathbb{N}_1} \notin \bigcup_{\lambda \in \Lambda}X_\lambda(E)\ \text{or}\ x^{(2)}|_{\widetilde{\mathbb{N}}_1} \notin \bigcup_{\lambda \in \Lambda}X_\lambda(E)
$$ 
and if $x^{(2)}|_{\mathbb{N}_1} \notin \bigcup_{\lambda \in \Lambda}X_\lambda(E)$ there is nothing to do. Assuming that $x^{(2)}|_{\widetilde{\mathbb{N}}_1} \notin \bigcup_{\lambda \in \Lambda}X_\lambda(E)$, we proceed in the same way above and consider a partition $\widetilde{\mathbb{N}}_1 = \mathbb{N}_2 \cup \widetilde{\mathbb{N}}_2$ for which $x^{(2)}|_{\mathbb{N}_2} \notin \bigcup_{\lambda \in \Lambda}X_\lambda(E)$. From this we have  \[x^{(1)}|_{\mathbb{N}_1 \cup \mathbb{N}_2}, x^{(2)}|_{\mathbb{N}_1 \cup \mathbb{N}_2} \notin \bigcup_{\lambda \in \Lambda}X_\lambda(E).\] 
We can continue this procedure to obtain $\mathbb{N}' = \bigcup_{i=1}^k \mathbb{N}_i$ ($k \le n$) and a partition $\mathbb{N} = \mathbb{N}' \cup \mathbb{N}''$ such that $x^{(1)}|_{\mathbb{N}'}, \dots , x^{(n)}|_{\mathbb{N}'} \notin \bigcup_{\lambda \in \Lambda}X_\lambda(E)$.
\end{proof}

A suitable notion of compatibility for maps will be useful in this section. Let $E$ and $F$ be Banach spaces and $f: E \to F$ be a map. Denote by $\widehat{f}$ the induced map $\widehat{f}: E^\mathbb{N} \to F^\mathbb{N}$, given by
\[
\widehat{f}\big( (x_j)_{j=1}^\infty \big) := \left( f(x_j) \right)_{j=1}^\infty \quad \text{for each } (x_j)_{j=1}^\infty \in E^\mathbb{N}.
\]
Adapting \cite[Definition 2.3]{Nogueira-Pellegrino} to a standard sequence class $X$, a map $f :E \to F$ with $f(0)=0$ is called \emph{compatible with $X$} if for every $(x_j)_{j=1}^\infty \in E^\mathbb{N}$ and all $\alpha \in \mathbb{K}\setminus\{0\}$,
\[
(f(x_j))_{j=1}^\infty \notin X(F) \Rightarrow (f(\alpha x_j))_{j=1}^\infty \notin X(F).
\]
We introduce a stronger notion: a map $f :E \to F$ with $f(0)=0$ is called \emph{linearly compatible with $X$} if it is compatible (with $X$) and satisfies the additional property that for any $(x_j)_{j=1}^\infty, (y_j)_{j=1}^\infty \in E^\mathbb{N}$,
\[
(f(x_j))_{j=1}^\infty, (f(y_j))_{j=1}^\infty \in X(F) \Rightarrow (f(x_j + y_j))_{j=1}^\infty \in X(F).
\]

Obvious consequences are listed next: if $f$ is compatible with $X$, then for all $(x_j)^\infty_{j=1} \in E^\mathbb{N}$ and $\alpha \in \mathbb{K}\setminus\{0\}$,
\begin{quote} 
\begin{enumerate}[i)]
%\item [i)] $(f(\alpha x_j))^\infty_{j=1} \in X(F)\Rightarrow (f( x_j))^\infty_{j=1} \in X(F)$,
\item $(f(x_j))^\infty_{j=1} \in X(F)\Rightarrow (f(\alpha x_j))^\infty_{j=1} \in X(F)$;
\item $(x_j)^\infty_{j=1} \in \widehat{f}^{-1}(X(F)) \Rightarrow \alpha(x_j)^\infty_{j=1} \in \widehat{f}^{-1}(X(F))$;
\end{enumerate}
\end{quote}
and if $f$ is linearly compatible with $X$ we obtain 
\begin{quote}
\begin{enumerate}
\item [iii)] $(x_j)^\infty_{j=1}, (y_j)^\infty_{j=1} \in \widehat{f}^{-1}(X(F)) \Rightarrow (x_j)^\infty_{j=1} + (y_j)^\infty_{j=1}\in \widehat{f}^{-1}(X(F))$ 
\end{enumerate}
\end{quote}
for all $(x_j)^\infty_{j=1}, (y_j)^\infty_{j=1} \in E^\mathbb{N}$. In other words, ii) and iii) tell us that $\widehat{f}^{-1}(X(F))$ is a linear subspace of $E^\mathbb{N}$ when $f$ is linearly compatible with $X$.

We recall a modern terminology from \cite{Botelho-Favaro} and \cite{Pellegrino-Raposo}. In what follows, for $X$ a standard sequence class, $f : E \to F$ an arbitrary map and $\{Y_\lambda\}_{\lambda \in \Lambda}$ a family of standard sequence classes, we will consider sets of the form
\[
G(X,f,\{Y_\lambda\}_{\lambda \in \Lambda})
:=
\left\lbrace x = (x_j)^\infty_{j=1} \in X(E) : \widehat{f}(x) = (f(x_j))^\infty_{j=1} \notin \bigcup_{\lambda \in \Lambda}Y_\lambda(F) \right\rbrace.
\]

Our upcoming results provide necessary conditions to obtain $(\alpha,\mathfrak{c})$-spaceability in sets of this nature. To this end, we require a condition that characterizes the behavior of a sequence $x \in G(X,f,Y)$ (that is, $\widehat{f}(x) \notin Y(F)$) in such a way that one can extract a subsequence whose image under $f$ lies in $Y(F)$. We refer to this property as \emph{detachable}. The formal definition, along with illustrative examples, is provided below.

\begin{definition}\rm
Let $X$ and $Y$ be standard sequence classes and $f:E \to F$ be a map linearly compatible with $Y$. A sequence $(x_j)_{j \in \mathbb{N}} \in G(X,f,Y)$ is said to be \emph{detachable} if there exists an infinite subset $\mathbb{N}_1 \subset \mathbb{N}$ whose complement $\mathbb{N} \setminus \mathbb{N}_1$ is also infinite and such that $(f(x_n))_{n \in \mathbb{N}_1} \in Y(F)$. We say that the triple $(X,f,Y)$ is detachable if all  $(x_j)_{j \in \mathbb{N}} \in G(X,f,Y)$ are detachable. The triple $(X,f,\{Y_\lambda\}_{\lambda \in \Lambda})$ is said to be detachable if $(X,f,Y_\lambda)$ is detachable for all ${\lambda \in \Lambda}$.
\end{definition}

\begin{remark}\rm
a) Clearly, the notion ``detachable'' is meaningful for a triple $(X,f,Y)$ only when $G(X,f,Y)$ is nonempty. Throughout, we therefore implicitly impose this nonemptiness condition.

\vskip 2mm

\noindent b) It follows immediately from  Proposition \ref{lema-restricoes} that if the triple $(X,f,Y)$ is detachable, then $(f(x_n))_{n \in \mathbb{N}_1} \in Y(F)$ and $(f(x_n))_{n \in \mathbb{N}_2} \notin Y(F)$ for all  $(x_j)_{j \in \mathbb{N}} \in G(X,f,Y)$, with $\mathbb{N}_2 = \mathbb{N} \setminus \mathbb{N}_1$.
\end{remark}

\begin{example}\label{ex-detach}\rm
Examples of detachable and non-detachable triples $(X,f,Y)$ are provided next.

\vskip 2mm

\noindent a) Let $p \in (0,\infty)$ and let $Id_p: \ell_p \to \ell_p$ be the identity operator in $\ell_p$. Then, $(\ell_{p^*}^w, Id_p, \ell_{p^*}(\cdot))$ is not detachable. Note that $(e_j)_{j=1}^\infty \in G(\ell_{p^*}^w, Id_p, \ell_{p^*}(\cdot)) = \ell_{p^*}^w(\ell_{p})\setminus \ell_{p^*}(\ell_{p})$. However, since $\|e_j\|_p = 1$ for every $j \in \mathbb{N}$, there is no subsequence of $(e_j)_{j=1}^\infty$ in  $\ell_{p^*}(\ell_p)$.

\vskip 2mm

\noindent b) Let $Y \in \{ \ell_p(\cdot),\ell_{p,q}(\cdot) ,\ell_p^w\}$  for $0<p,q<\infty$ or $Y=\ell_M(\cdot)$ for an Orlicz function $M$. If $f:E \to F$ is a map linearly compatible with $Y$ and  continuous at $0$, $X$ is a standard sequence class such that $X(E) \subseteq c_0(E)$ and $G(X,f,Y) \neq \emptyset$, then $(X,f,Y)$ is detachable. To prove our assertion, first note that for every $(x_j)^\infty_{j=1} \in X(E)$ we have $(f(x_j))^\infty_{j=1} \in c_0(F)$. Thus,\\
i) The case $Y= \ell_p(\cdot)$ follows the lines of the scalar case proved in \cite[Theorem 2.2]{Favaro-Pellegrino-Tomaz} as well as the case $Y = \ell_{p,q}(\cdot)$. Both will therefore be omitted. \\
ii) For $Y = \ell_p^w$ with $0<p<\infty$, consider $(x_j)_{j=1}^\infty \in G(X,f,\ell_p^w)$. Since $(f(x_j))_{j=1}^\infty \in c_0(F)$, there exists a partition  $\mathbb{N}=  \mathbb{N}_1 \cup \mathbb{N}_2$ such that
$$
|\varphi(f(x_{k}))| \leq \|\varphi\|\|f(x_{k})\| < \|\varphi\|\frac{1}{2^k},\quad \forall\, \varphi \in F',\, \forall\, k \in \mathbb{N}_1.
$$
Thus, for all $\varphi \in F'$, we have
$$
\sum_{k \in \mathbb{N}_1} |\varphi(f(x_{k}))|^p \leq \|\varphi\|\sum_{k \in \mathbb{N}_1} \left(\frac{1}{2^k}\right)^p < \infty,
$$
which implies that  $(f(x_{k}))_{k \in \mathbb{N}_1} \in \ell_p^w(F)$.\\
iii) Let $M$ be an Orlicz function and let $(x_j)^\infty_{j=1} \in G(X,f,\ell_M(\cdot))$. As we have $(f(x_j))^\infty_{j=1} \in c_0(F) \setminus \ell_M(F)$, there exists a partition $\mathbb{N}=  \mathbb{N}_1 \cup \mathbb{N}_2$ such that
$$
\|f(x_{k})\| < \frac{1}{2^k}, \ k \in \mathbb{N}_1.
$$
Since $M$ is a convex function, we have
$$
M(\alpha x + (1-\alpha) y ) \leq \alpha M(x) + (1-\alpha)M(y)
$$
for all $x,y \in [0,\infty)$ and $\alpha \in [0,1]$. Given an arbitrary real number \(\eta >0 \) and choosing
$$
\alpha = \frac{1}{2^k},\ x = \frac{1}{\eta}\ \text{and}\ y = 0,
$$
we obtain
$$
M\left( \frac{1}{2^k \eta}\right) \leq \frac{1}{2^k} M\left( \frac{1}{\eta}\right),\ \forall k \in \mathbb{N}_1.
$$
Thus
$$
\sum_{k \in \mathbb{N}_1 }M\left(\frac{\|f(x_{k})\|}{\eta} \right) \leq \sum_{k \in \mathbb{N}_1}M\left(\frac{1}{2^k\eta} \right) \leq \sum_{k \in \mathbb{N}_1}^{\infty}\frac{1}{2^k} M\left( \frac{1}{\eta}\right) < \infty
$$
and we conclude that $(f(x_{k}))_{k \in \mathbb{N}_1} \in \ell_M(F)$.
\end{example}

\begin{proposition}
Let $X, Y_1$ and $Y_2$ be standard sequence classes and let $f:E\to F$ be a linearly compatible map with $Y_1$ and $Y_2$. If $(X,f,Y_1)$ is detachable and $Y_1(E) \subseteq Y_2(E)$, then $(X,f,Y_2)$ is detachable.
\end{proposition}

For example, since $\ell_p(E) \subseteq \ell_p^u(E)$ and $\ell_p(E) \subseteq \ell_q(E)$ for all Banach space $E$ an all $p \le q \in (0, \infty]$, it follows that the triples $(X,f,\ell_p^u)$ and $(X,f,\ell_q(\cdot))$ are detachable for every standard sequence class $X$, whenever $(X,f,\ell_p(\cdot))$ is detachable.

The detachable notion is compatible with a finite number of sequences as demonstrated in the following proposition.

\begin{proposition} \label{lema-restricoes-in}
Let $X$ be a standard sequence class, $\{Y_\lambda\}_{\lambda \in \Lambda}$ be an $F$-nested family of standard sequence classes,  $f:E \to F$ be a map linearly compatible with every $Y_\lambda,\, \lambda \in \Lambda$, and $x^{(1)},\dots, x^{(n)} \in G(X,f,\{Y_\lambda\}_{\lambda \in \Lambda})$. Suppose that the triple $(X,f,Y_{\widetilde{\lambda}})$ is detachable for some $\widetilde{\lambda} \in \Lambda$. Then there exists a partition  $\mathbb{N} = \mathbb{N}' \cup \mathbb{N}''$ such that
\[
\widehat{f}(x^{(k)})|_{\mathbb{N}'} \in \bigcup_{\lambda \in \Lambda}Y_\lambda(F)
\quad \text{and} \quad
\widehat{f}(x^{(k)})|_{\mathbb{N}''} \notin \bigcup_{\lambda \in \Lambda}Y_\lambda(F),
\quad \text{for all} \ k=1,\dots,n.
\]
\end{proposition}

\begin{proof}
First we apply Proposition \ref{lema-restricoes} to obtain a partition $\mathbb{N} = \widetilde{\mathbb{N}}_1 \cup (\mathbb{N}\setminus\widetilde{\mathbb{N}}_1)$ that fulfills
\begin{equation*}
\widehat{f}(x^{(1)})|_{\mathbb{N}\setminus\widetilde{\mathbb{N}}_1}, \dots , \widehat{f}(x^{(n)})|_{\mathbb{N}\setminus\widetilde{\mathbb{N}}_1} \notin \bigcup_{\lambda \in \Lambda}Y_\lambda(F).
\end{equation*}
By the detachment of \((X, f, Y_{\widetilde{\lambda}})\), there exists a partition \(\widetilde{\mathbb{N}}_1 = \mathbb{N}_1 \cup \mathbb{N}_1'\) such that  
\[
\widehat{f}(x^{(1)})|_{\mathbb{N}_1} \in Y_{\widetilde{\lambda}}(F).
\]
Indeed, if \(\widehat{f}(x^{(1)})|_{\widetilde{\mathbb{N}}_1} \notin Y_{\widetilde{\lambda}}(F)\), one applies the detachable property of \((X, f, Y_{\widetilde{\lambda}})\); otherwise, one may simply take \(\widetilde{\mathbb{N}}_1 = \mathbb{N}_1\). Using the same argument for $\widehat{f}(x^{(2)})$, there exists a partition $\mathbb{N}_1 = \mathbb{N}_2 \cup \mathbb{N}_2'$ such that
$$
\widehat{f}(x^{(1)})|_{\mathbb{N}_2}, \widehat{f}(x^{(2)})|_{\mathbb{N}_2} \in Y_{\widetilde{\lambda}}(F).
$$
Proceeding in this manner, we can take $\mathbb{N}' := \mathbb{N}_n$ (if necessary) satisfying
$$
\widehat{f}(x^{(k)})|_{\mathbb{N}'} \in Y_{\widetilde{\lambda}}(F) \subseteq \bigcup_{\lambda \in \Lambda}Y_\lambda(F)
\quad \text{for all} \ k=1,\dots,n.
$$ 
and, setting $\mathbb{N}'' := \mathbb{N}\setminus \mathbb{N}' = (\mathbb{N}\setminus\widetilde{\mathbb{N}}_1) \cup \left(\bigcup_{i=1}^n\mathbb{N}_{i}' \right)$, where $k \leq n$, we have
$$
\widehat{f}(x^{(k)})|_{\mathbb{N}''} \notin \bigcup_{\lambda \in \Lambda}Y_\lambda(F)
\quad \text{for all} \ k=1,\dots,n.
$$
\end{proof}

\section{Geometric spaceability in standard sequence spaces} \label{Sec-LIN-SSS}

We begin by recalling the basic notions of lineability needed in this section. For further details, see the seminal works \cite{Lineability-book,Aron-Gurariy-Seoane,Gurariy-Quarta}, as well as the geometric notions introduced in \cite{Favaro-Pellegrino-Tomaz,Pellegrino-Raposo}. Let $\mu$ be a cardinal number and $V$ a linear space. A subset $A \subset V$ is \emph{$\mu$-lineable} if $A \cup \{0\}$ contains a $\mu$-dimensional linear subspace of $V$. If $V$ is a topological vector space and $A \cup \{0\}$ contains a closed $\mu$-dimensional subspace, then $A$ is \emph{$\mu$-spaceable}. When $\dim(V)=\mu$, we say that $A$ is \emph{$\mu$-maximal lineable} or \emph{$\mu$-maximal spaceable}, respectively.

Given cardinals $\alpha \leq \beta$, the set $A$ is \emph{$(\alpha,\beta)$-lineable (spaceable)} if it is $\alpha$-lineable and, for every 
$\alpha$-dimensional subspace $U_\alpha \subset A \cup \{0\}$, there exists a $\beta$-dimensional (closed) subspace $U_\beta$ of $V$ such that
\(
U_\alpha \subset U_\beta \subset A \cup \{0\}.
\)
Moreover, $A$ is \emph{pointwise $\alpha$-lineable (spaceable)} if, for every $x \in A$, there exists an $\alpha$-dimensional (closed) 
subspace $W_x$ with
\(
x \in W_x \subset A \cup \{0\}.
\)

The next lemma generalizes a previous result established for the case of standard sequence classes (see \cite[Lemma 2.1]{Favaro-Pellegrino-Tomaz}).

\begin{lemma}\label{lemma-LI}
If $X$ is a standard sequence class and $\{x^{(1)},\dots,x^{(n)}\}$ is a linearly independent set in $X(E)$, then there exists $n_0 \in \mathbb{N}$ such that the set $\{(x^{(k)}_1,\dots, x^{(k)}_{n_0}): k=1, \ldots, n\}$ is linearly independent in $E^{n_0}$.
\end{lemma}
\begin{proof}
Suppose, contrary to our claim, that for each $j \in \mathbb{N}$ there exist $a^{(1)}_j,\dots,a^{(n)}_j \in \mathbb{K}\setminus \{0\}$ such that
\begin{equation}\label{hip-cont}
	a^{(1)}_jx^{(1)}+\dots+a^{(n)}_jx^{(n)} = (0,\dots,0,z_{j+1},z_{j+2},\dots).
\end{equation}
Without loss of generality, we can take $|a_j^{(k)}| =1$ for all $j \in \mathbb{N}, k=1,\ldots ,n$ and define $a_j := (a^{(1)}_j,\dots,a^{(n)}_j) \in \mathbb{K}^n$. Since $(a_j)_{j=1}^\infty$ is bounded, there exists a subsequence $(a_{j_k})_{k=1}^\infty$  of $(a_j)_{j=1}^\infty$ such that
$$
(a^{(1)}_{j_k},\dots,a^{(n)}_{j_k}) = a_{j_k} \stackrel{k}{\longrightarrow} a = (a^{(1)},\dots,a^{(n)})
$$
and, as $\|a \|_\infty = \|a_{j_k}\|_\infty= 1$, we have $a\neq 0$. Furthermore, it follows that
\begin{equation}\label{conv-subseq}
a^{(1)}_{j_k}x^{(1)}+\dots+a^{(n)}_{j_k}x^{(n)} \stackrel{k}{\longrightarrow}  a^{(1)}x^{(1)}+\dots+a^{(n)}x^{(n)}.
\end{equation}
On the other hand, due to \eqref{hip-cont}, we have
\begin{equation}\label{pi-m-0}
\pi_m(a^{(1)}_{j_k}x^{(1)}+\dots+a^{(n)}_{j_k}x^{(n)} ) = a^{(1)}_{j_k}x^{(1)}_m+\dots+a^{(n)}_{j_k}x^{(n)}_m = 0 
\end{equation}
for all $m \le j_k$ and for sufficiently large values of $k$. Since the projections $\pi_m$ are continuous, from \eqref{conv-subseq} and \eqref{pi-m-0}, we obtain
\begin{align*}
\pi_m(a^{(1)}x^{(1)}+\dots+a^{(n)}x^{(n)}) &= \pi_m(\lim_{k} a^{(1)}_{j_k}x^{(1)}+\dots+a^{(n)}_{j_k}x^{(n)})\\
&= \lim_{k}\pi_m(a^{(1)}_{j_k}x^{(1)}+\dots+a^{(n)}_{j_k}x^{(n)} ) =0
\end{align*}
for all $m \in \mathbb{N}$. It follows that $a^{(1)}x^{(1)}+\dots+a^{(n)}x^{(n)} = 0$ and as $\{x^{(1)},\dots,x^{(n)}\}$ is linearly independent we have $a=0$, a contradiction.
\end{proof}

Next we state our main result concerning the geometric lineability approach on the sets $G(X,f,\{Y_\lambda\}_{\lambda \in \Lambda})$, which sharpens Theorems 3.3 and 3.5 in \cite{Pellegrino-Raposo}.

\begin{theorem} \label{th-standardseq}
Let $X$ be a standard sequence class, $\{Y_\lambda\}_{\lambda \in \Lambda}$ be an $F$-nested family of standard sequence classes and let $f:E \to F$ be a map that is linearly compatible with $Y_\lambda$ for all $\lambda \in \Lambda$. If the triple $(X,f,\{Y_\lambda\}_{\lambda \in \Lambda})$ is detachable, then
\[G(X,f,\{Y_\lambda\}_{\lambda \in \Lambda})
\ \text{is} \
(\alpha, \mathfrak{c})\text{-spaceable} 
\ \text{if and only if} \ \alpha < \aleph_0.
\]
\end{theorem}

\begin{proof}
We begin by proving that the case $\alpha = \aleph_0$ is impossible. Let $\mathbb{N} = \bigcup_{j=1}^\infty \mathbb{N}_j$ be a partition, and fix $x \in E$. For each $k \in \mathbb{N}$, define a sequence $x^{(k)} := \left(x,x_j^{(k)}, x_{j+1}^{(k)}, \ldots\right)_{j \geq 2}$ such that 
\[
x_j^{(k)} = 0\ \text{if}\  j \notin \mathbb{N}_k \ \text{and}\ (x_j^{(k)})_{j \in \mathbb{N}_k (j\geq2)} \in G(X,f,\{Y_\lambda\}_{\lambda \in \Lambda}).
\]
We can assume that 
$$
\|(x_j^{(k)})_{j \in \mathbb{N}_k (j\geq2)}\|_{X(E)}= 2^{-k} \ \text{for all}\ k \in \mathbb{N}
$$
and define
$$
\Omega := \operatorname{span}\{x^{(k)}: k \in \mathbb{N}\}.
$$
We now show that $\Omega\setminus\{0\} \subset G\left(X,f,\{Y_\lambda\}_{\lambda \in \Lambda}\right)$. Let $\bar{x} = \sum_{i=1}^n b_i x^{(k_i)} \in \Omega \setminus \{0\}$ with all $b_i \neq 0$. Suppose, for contradiction, that $(f(\bar{x}_j))_{j \in \mathbb{N}} \in Y_{\lambda'}(F)$ for some $\lambda' \in \Lambda$. Then, restricting to the indices in $\mathbb{N}_{k_1} \setminus \{1\}$, we have
\[
(f(b_1 x_j^{(k_1)}))_{j \in \mathbb{N}_{k_1},\, j\geq 2} = (f(\bar{x}_j))_{j \in \mathbb{N}_{k_1},\, j\geq 2} \in Y_{\lambda'}(F).
\]
Since $f$ is compatible with $Y_{\lambda'}$, it follows that $(f(x_j^{(k_1)}))_{j \in \mathbb{N}_{k_1},\, j\geq 2} \in Y_{\lambda'}(F) \subseteq \bigcup_{\lambda \in \Lambda} Y_\lambda(F)$. This contradicts the fact that $(x_j^{(k_1)})_{j \in \mathbb{N}_{k_1},\, j\geq 2} \in G(X, f, \{Y_{\lambda}\}_{\lambda \in \Lambda})$. Hence, $\Omega \setminus \{0\} \subset G(X, f, \{Y_{\lambda}\}_{\lambda \in \Lambda})$.

Next, we prove that no closed subspace $\Theta$ of $X(E)$ containing $\Omega$ can satisfy $\Theta \setminus \{0\} \subset G(X, f, \{Y_{\lambda}\}_{\lambda \in \Lambda})$. Suppose, for contradiction, that such a $\Theta$ exists. Then,
\begin{align*}
\lim_{k} \| x^{(k)} - x \cdot e_1 \|_{X(E)}
= \lim_{k} \left\| (0, x_j^{(k)}, x_{j+1}^{(k)}, \ldots) \right\|_{X(E)}
= \lim_{k} 2^{-k} = 0.
\end{align*}
Therefore, $x \cdot e_1 \in \overline{\Omega} \subset \overline{\Theta} = \Theta \subset G(X, f, \{Y_{\lambda}\}_{\lambda \in \Lambda}) \cup \{0\}$. However, this implies $f(x) \cdot e_1 \notin Y_{\lambda}(F)$ for all $\lambda \in \Lambda$, which is impossible because $c_{00}(F) \subset Y_{\lambda}(F)$ for every $\lambda$. This concludes the proof for $\alpha = \aleph_0$.

\vskip 3mm

We now prove the result for $\alpha = 3$. The argument, with obvious modifications, applies to any finite $\alpha = n \in \mathbb{N}$. Let $\{x, y, z\}$ be a linearly independent subset of $X(E)$ such that
\begin{equation}\label{W-notin}
W := \operatorname{span}\{x, y, z\} \subset G(X, f, \{Y_{\lambda}\}_{\lambda \in \Lambda}) \cup \{0\}.
\end{equation}
By Proposition~\ref{lema-restricoes}, there exists an infinite set $\mathbb{N}_1 := \{\alpha_1 < \alpha_2 < \cdots \} \subset \mathbb{N}$, with infinite complement, such that
\begin{equation}\label{xyz-n1-G}
	x|_{\mathbb{N} \setminus \mathbb{N}_1},\ y|_{\mathbb{N} \setminus \mathbb{N}_1},\ z|_{\mathbb{N} \setminus \mathbb{N}_1} \in G(X, f, \{Y_{\lambda}\}_{\lambda \in \Lambda}).
\end{equation}
Rewriting these restrictions yields:
\[
\begin{array}{l}
\left( \sum_{j \notin \mathbb{N}_1}x_j \cdot e_j \right)|_{\mathbb{N}\setminus\mathbb{N}_1}
= x|_{\mathbb{N}\setminus\mathbb{N}_1}\in  G(X,f,\{Y_\lambda\}_{\lambda \in \Lambda}), \\ %\vskip1mm
\left( \sum_{j \notin \mathbb{N}_1}y_j \cdot e_j\right)|_{\mathbb{N}\setminus\mathbb{N}_1}
= y|_{\mathbb{N}\setminus\mathbb{N}_1}\in  G(X,f,\{Y_\lambda\}_{\lambda \in \Lambda}), \ \text{and} \\
\left( \sum_{j \notin \mathbb{N}_1}z_j \cdot e_j \right)|_{\mathbb{N}\setminus\mathbb{N}_1}
= z|_{\mathbb{N}\setminus\mathbb{N}_1} \in  G(X,f,\{Y_\lambda\}_{\lambda \in \Lambda}).
\end{array}
\]
Since the sets $G(X, f, \{Y_{\lambda}\}_{\lambda \in \Lambda})$ have the \emph{spreading zeros} property (Definition~\ref{standard-sc}-(iii)), it follows that
\[
\sum_{j \notin \mathbb{N}_1}x_j \cdot e_j,\, \sum_{j \notin \mathbb{N}_1}y_j \cdot e_j,\, \sum_{j \notin \mathbb{N}_1}z_j \cdot e_j \in  G(X,f,\{Y_\lambda\}_{\lambda \in \Lambda}).
\]
each $(X, f, Y_\lambda)$ is detachable for all $\lambda \in \Lambda$, Proposition~\ref{lema-restricoes-in} allows us to assume, without loss of generality, that
\[
(f(x_j))_{j \in \mathbb{N}_1},\ (f(y_j))_{j \in \mathbb{N}_1},\ (f(z_j))_{j \in \mathbb{N}_1} \in Y_{\lambda'}(F)
\]
for some $\lambda' \in \Lambda$. From Lemma \ref{lemma-LI}, there exists $n_0 \in \mathbb{N}$ such that the vectors formed by the first $n_0$-coordinates of $x,\, y$ and $z$ are linearly independent in $E^{n_0}$. Set $\mathbb{O}:= \mathbb{N}_1\setminus\{1,\dots, n_0\}$, so, $\mathbb{N} \setminus \mathbb{O} = \left( \mathbb{N}\setminus \mathbb{N}_1\right)\cup \{1,\dots,n_0\}$. Then
\begin{equation}\label{eqO}
\sum_{j \notin \mathbb{O}}x_j \cdot e_j,\, \sum_{j \notin \mathbb{O}}y_j \cdot e_j,\, \sum_{j \notin \mathbb{O}}z_j \cdot e_j \in  G(X,f,\{Y_\lambda\}_{\lambda \in \Lambda})
\end{equation}
and
\begin{equation} \label{Ylamba'}
(f(x_j))_{j\in\mathbb{O}},\ (f(y_j))_{j\in\mathbb{O}},\ (f(z_j))_{j\in\mathbb{O}}
\in Y_{\lambda'}(F)\subseteq\bigcup_{\lambda\in\Lambda}Y_\lambda(F).
\end{equation}
Take an infinite partition
\[
\mathbb{O}:=\bigcup_{i=1}^\infty\mathbb{O}_i,
\qquad
\mathbb{O}_i:=\{n^{(i)}_1<n^{(i)}_2<\cdots\}.
\] 
Since $\mathbb{N}\setminus \mathbb{O}$ is infinite, choose an injective map $h: \mathbb{N}\to \mathbb{N}$ with $\mathbb{N}\setminus \mathbb{O} = h(\mathbb{N})$. As $x$ lies in $X(E)$, define for each $i \in \mathbb{N}$
$$
\varepsilon^{(i)} := \sum_{j=1}^{\infty}x_{h(j)} \cdot e_{n^{(i)}_j} \in X(E).
$$
Moreover
$$
\|\varepsilon^{(i)}\|_{X(E)} = \|(x_{h(j)})_{j \in \mathbb{N}} \|_{X(E)} \leq \|x\|_{X(E)}.
$$

Since $X(E)$ is quasi-Banach, it is a $p$-Banach space for some $p \in (0,1]$. For any arbitrary sequence $(a_i)^\infty_{i=1} \in \ell_p$,
$$
\sum_{i=1}^\infty\|a_i \varepsilon^{(i)}\|_{X(E)}^p
= \sum_{i=1}^\infty|a_i|^p\| \varepsilon^{(i)}\|_{X(E)}^p
\leq \| x\|_{X(E)}^p \sum_{i=1}^\infty|a_i|^p < \infty,
$$
hence the series \(\sum_{i=3}^\infty a_i\varepsilon^{(i)}\) converges in \(X(E)\). Define $T: \ell_p \to X(E)$ by
$$
T((a_i)^\infty_{i=0})
:= a_0x +a_1y+ a_2z + \sum_{i=3}^\infty a_i \varepsilon^{(i)}.
$$
The operator $T$ is bounded and injective. We prove the injectivity property: let $a=(a_i)^\infty_{i=0} \in \ell_p$ satisfying
$$
a_0x +a_1 y +a_2 z + \sum_{i=3}^\infty a_i \varepsilon^{(i)} = T(a) = 0.
$$
With $A:=\{1, \dots, n_0\} \subset \mathbb{N}\setminus \mathbb{O}$, the first $n_0$ coordinates of every $\varepsilon^{(i)}$ vanish, hence
$$
T(a)|_A = a_0 (x_j)_{j=1}^{n_0} + a_1 (y_j)_{j=1}^{n_0}+ a_2 (z_j)_{j=1}^{n_0} = 0.
$$
By linear independence of these vectors we obtain \(a_0=a_1=a_2=0\) and so $\sum_{i=3}^\infty a_i\varepsilon^{(i)} = 0$. Since the supports of the \(\varepsilon^{(i)}\) are pairwise disjoint, it follows that \(a_i=0\) for every \(i\ge 3\). Thus \(T\) is injective.

\vskip 3mm

\noindent\textbf{Claim 1.} $T(\ell_p)\setminus \{0\} \subseteq G(X,f,\{Y_\lambda\}_{\lambda \in \Lambda})$.

\vskip 2mm

Fix $a=(a_i)_{i=0}^\infty \in \ell_p$. The argument proceeds by considering two separate cases. First we suppose that $a_0=a_1=a_2=0.$ Since $\varepsilon^{(i)}$ have pairwise disjoint supports, we have
$$
\widehat{f}(T(a)) = \sum_{i=3}^\infty \widehat{f}(a_i \varepsilon^{(i)}). 
$$
On the other hand, as $\varepsilon^{(i)}|_{\mathbb{N}\setminus \mathbb{O}_i} = 0$, we have
$$
\widehat{f}(T(a))|_{\mathbb{O}_{k}}
= \sum_{i=3}^\infty \widehat{f}(a_i \varepsilon^{(i)})|_{\mathbb{O}_{k}}
= \widehat{f}(a_{k} \varepsilon^{(k)})|_{\mathbb{O}_{k}}
= \widehat{f}(a_{k} \varepsilon^{(k)}|_{\mathbb{O}_{k}})
= (f(a_{k}x_{h(j)}))_{j \in \mathbb{N}}, \ k \in \mathbb{N}.
$$
Suppose that $\widehat{f}(T(a)) \in Y_{\tilde{\lambda}}(F)$ for some $\tilde{\lambda} \in \Lambda$. Then
$$
(f(a_{k}x_{h(j)}))_{j \in \mathbb{N}} \in Y_{\tilde{\lambda}}(F)
$$
and, since $f$ is linearly compatible with $Y_{\tilde{\lambda}}$, we have $(f(x_{h(j)}))_{j \in \mathbb{N}} \in Y_{\tilde{\lambda}}(F)$, which yields
\[
\sum_{j \notin \mathbb{O}}f(x_j) \cdot e_j \in Y_{\tilde{\lambda}}(F) \subseteq \bigcup_{\lambda \in \Lambda}Y_\lambda(F)
\]
contradicting \eqref{eqO}. So $T(a) \in G(X,f,\{Y_\lambda\}_{\lambda \in \Lambda})$.\\

\vskip 2mm

We now treat the case in which $a_i \neq 0$ for some $i \in \{0,1,2\}$. Since $\varepsilon^{(i)}|_{\mathbb{N}\setminus \mathbb{O}} = 0$ for all $i \in \mathbb{N}$, it follows that
$$
T(a)|_{\mathbb{N} \setminus \mathbb{O}} =\left. \left( a_0x +a_1y+ a_2z + \sum_{i=3}^\infty a_i \varepsilon^{(i)}\right) \right|_{\mathbb{N} \setminus \mathbb{O}} = (a_0x +a_1y+ a_2z)|_{\mathbb{N} \setminus \mathbb{O}}.
$$
Consequently
$$
\widehat{f}(T(a)|_{\mathbb{N} \setminus \mathbb{O}}) =(f(a_0x_j +a_1y_j+ a_2z_j))_{j\in \mathbb{N}\setminus\mathbb{O}}.
$$
By \eqref{W-notin} and the Proposition \ref{lema-restricoes}-(i), we obtain
\begin{equation} \label{fa012inornot}
(f(a_0x_j +a_1y_j+ a_2z_j))_{j\in \mathbb{N}\setminus\mathbb{O}} \notin \bigcup_{\lambda \in \Lambda}Y_\lambda(F)
\ \textrm{or} \
(f(a_0x_j +a_1y_j+ a_2z_j))_{j\in \mathbb{O}} \notin \bigcup_{\lambda \in \Lambda}Y_\lambda(F).
\end{equation}
However, from \eqref{Ylamba'} and the linear compatibility of $f$ with $Y_{\lambda'}$, it follows that $(f(a_0x_j))_{j \in \mathbb{O}}$,  $(f(a_1y_j))_{j \in \mathbb{O}}$, $(f(a_2z_j))_{j \in \mathbb{O}} \in Y_{\lambda'}(F)$, and this yields
\[
(f(a_0x_j +a_1y_j+ a_2z_j))_{j \in \mathbb{O}} \in \bigcup_{\lambda \in \Lambda}Y_\lambda(F).
\]
Hence, the first alternative of \eqref{fa012inornot} is valid, yielding $\widehat{f}(T(a)) \notin \bigcup_{\lambda \in \Lambda}Y_\lambda(F)$. This establishes the claim.

\vskip 3mm

The first claim established the $(3, \mathfrak{c})$-lineability of $G(X,f,\{Y_\lambda\}_{\lambda \in \Lambda})$. We now prove the spaceability assertion.

\vskip 3mm

\noindent\textbf{Claim 2.} $\overline{T(\ell_p)}\setminus\{0\} \subseteq G(X,f,\{Y_\lambda\}_{\lambda \in \Lambda})$.

\vskip 2mm

Let $w = (w_j)_{j=1}^\infty \in \overline{T(\ell_p)}\setminus\{0\}$, and sequences $b^{(k)} =(b_i^{(k)})_{i=0}^\infty \in \ell_p$, $k \in \mathbb{N}$ such that $w = \lim_k T(b^{(k)})$ in $X(E)$. Since
$$
T(b^{(k)})|_{\mathbb{N} \setminus \mathbb{O}} = (b_0^{(k)}x_j + b_1^{(k)}y_j + b_2^{(k)}z_j)_{j \in \mathbb{N} \setminus \mathbb{O}},
$$
the limit
\begin{equation*}
\lim_k (b_0^{(k)}x_j + b_1^{(k)}y_j + b_2^{(k)}z_j)_{j \in \mathbb{N} \setminus \mathbb{O}}
\end{equation*}
exists in $X(E)$. As before, since $\{1,\dots, n_0\} \subset \mathbb{N} \setminus \mathbb{O}$, the set
$$
\{\widetilde{x} =(x_j)_{j \in \mathbb{N}\setminus\mathbb{O}}, \widetilde{y}=(y_j)_{j \in \mathbb{N}\setminus\mathbb{O}}, \widetilde{z}= (z_j)_{j \in \mathbb{N}\setminus\mathbb{O}}\}
$$
is a linearly independent set in $X(E)$. Because
$
\widetilde{W}:=\operatorname{span} \{\widetilde{x},\widetilde{y},\widetilde{z}\}
$
is closed, there exist $b_0, b_1,b_2 \in \mathbb{K}$ such that
\begin{align*}
\lim_k (b_0^{(k)}x_j + b_1^{(k)}y_j + b_2^{(k)}z_j)_{j \in \mathbb{N} \setminus \mathbb{O}}
&= (b_0x_j + b_1y_j + b_2z_j)_{j \in \mathbb{N} \setminus \mathbb{O}}\\
&= b_0 \widetilde{x} + b_1 \widetilde{y} + b_2 \widetilde{z}.
\end{align*}
Consider the continuous linear functional $\varphi_{\widetilde{x}}: \widetilde{W} \to \mathbb{K}$ defined by $\varphi_{\widetilde{x}} (\alpha \widetilde{x} + \beta \widetilde{y} + \gamma \widetilde{z}) = \alpha$. Then
\begin{align*}
b_0
&= \varphi_{\widetilde{x}} (b_0 \widetilde{x} + b_1 \widetilde{y} + b_2 \widetilde{z})\\
&= \varphi_{\widetilde{x}} \left[ \lim_k (b_0^{(k)} \widetilde{x} + b_1^{(k)} \widetilde{y} + b_2^{(k)} \widetilde{z})\right] \\
&= \lim_k \varphi_{\widetilde{x}}(b_0^{(k)} \widetilde{x} + b_1^{(k)} \widetilde{y} + b_2^{(k)} \widetilde{z})
=\lim_k b_0^{(k)}
\end{align*}
Similarly, $\lim_k b_1^{(k)} = b_1$ and $\lim_k b_2^{(k)} = b_2$. Hence, we can write
$$
w_j = b_0 x_j + b_1 y_j +b_2 z_j,\ \ j \in \mathbb{N}\setminus\mathbb{O}.
$$

Suppose first that $b_i = \lim_k b_i^{(k)} = 0$ for all $i \in \{0,1,2\}$. Then $w_j = 0$ for all $j \in \mathbb{N}\setminus\mathbb{O}$. Since $w\neq 0$, there is $r \in \mathbb{O}$ with $w_r \neq 0$. As $\mathbb{O} = \bigcup_{i=1}^\infty\mathbb{O}_i$, there exist $m, t \in \mathbb{N}$ such that $e_r = e_{n_t^{(m)}}$. Note that
\[
T(a^{(k)})
= b_0^{(k)}x +b_1^{(k)}y+ b_2^{(k)}z +
\sum_{i=3}^\infty \sum_{j=1}^\infty b_i^{(k)} x_{h(j)} \cdot e_{n_j^{(i)}},
\ k \in \mathbb{N},
\]
and the $r$-th coordinate of $T(a^{(k)})$ is
$$
b_0^{(k)}x_r +b_1^{(k)}y_r+ b_2^{(k)}z_r + b_m^{(k)}x_{h(t)}.
$$
Since $X(E) \stackrel{C}{\hookrightarrow} \ell_\infty(E)$, we have 
$$
0 \neq w_r = \lim_k\left(b_0^{(k)}x_r +b_1^{(k)}y_r+ b_2^{(k)}z_r + b_m^{(k)}x_{h(t)} \right)=\left( \lim_k b_m^{(k)}\right) x_{h(t)}
$$
and
\begin{align*}
w|_{\mathbb{O}_m}
&= \left.\left[ \lim_k\left( b_0^{(k)}x +b_1^{(k)}y+ b_2^{(k)}z + \sum_{i=3}^\infty b_i^{(k)}\varepsilon^{(i)} \right)  \right]\right |_{\mathbb{O}_m} 
= \left.\left[ \lim_k\left( \sum_{i=3}^\infty b_i^{(k)}\varepsilon^{(i)} \right)  \right]\right|_{\mathbb{O}_m}\\
&= \lim_k \left. \left( \sum_{i=3}^\infty b_i^{(k)}\varepsilon^{(i)} \right)\right |_{\mathbb{O}_m}
= \lim_k  \left( \sum_{i=3}^\infty b_i^{(k)}\varepsilon^{(i)}|_{\mathbb{O}_m} \right)
= \lim_k \left( b_m^{(k)} \varepsilon^{(m)}|_{\mathbb{O}_m}\right)\\
&=\left(\lim_k  b_m^{(k)} \right)\cdot(x_j)_{j \in \mathbb{N}\setminus \mathbb{O}}.
\end{align*}

If $\widehat{f}(w) \in Y_{\widehat{\lambda}}(E)$ for some $\widehat{\lambda} \in \Lambda$, by definition of standard sequence classes,
\begin{align*}
\left( f\left( \left(\lim_k  b_m^{(k)} \right)\cdot x_j\right)\right)_{j \in \mathbb{N}\setminus \mathbb{O}}& = \widehat{f}\left( \left(\lim_k  b_m^{(k)} \right)\cdot(x_j)_{j \in \mathbb{N}\setminus \mathbb{O}}\right)\\
&=\widehat{f}(w|_{\mathbb{O}_m})\\
&=\widehat{f}(w)|_{\mathbb{O}_m} \in  Y_{\widehat{\lambda}}(F).
\end{align*}
Since $f$ is linearly compatible, it follows that
$$
(f(x_j))_{j \in \mathbb{N}\setminus \mathbb{O}} \in Y_{\widehat{\lambda}}(F) \subseteq \bigcup_{\lambda \in \Lambda}Y_\lambda(F),
$$ 
contradicting \eqref{xyz-n1-G}.

Now suppose that $b_i = \lim_k b_i^{(k)} \neq 0$ for some $i \in \{0,1,2\}$. If $\widehat{f}(w) \in Y_{\widehat{\lambda}}(F)$ for some $\widehat{\lambda} \in \Lambda$, then
$$
\widehat{f} \left((b_0x + b_1y + b_2z)|_{\mathbb{N}\setminus \mathbb{O}} \right)
%= (f(b_0x_j + b_1y_j + b_2z_j))_{j \in \mathbb{N}\setminus\mathbb{O}}
= \widehat{f}(w|_{\mathbb{N}\setminus \mathbb{O}})
= \widehat{f}(w)|_{\mathbb{N}\setminus \mathbb{O}} \in Y_{\widehat{\lambda}}(F) \subseteq \bigcup_{\lambda \in \Lambda}Y_\lambda(F)
$$
which is impossible by the argument from Claim 1 (see \eqref{fa012inornot}).

In both cases $\widehat{f}(w) \notin \bigcup_{\lambda \in \Lambda}Y_\lambda(F)$ and so $
\overline{T(\ell_p)}\setminus\{0\} \subseteq G(X,f,\{Y_\lambda\}_{\lambda \in \Lambda})$. This proves the claim and completes the proof.
\end{proof}

\section{Pointwise spaceability in operator ideals}\label{Sec-PL-OI}

In this section, we establish a pointwise spaceability result in the context of operator ideals. We begin by recalling the notion of quasi-Banach operator ideals. An \emph{operator ideal} $\mathcal{I}$ is a subclass of $\mathcal{L}$, the class of all continuous linear operators between Banach spaces, such that for all Banach spaces $E$ and $F$, the components $\mathcal{I}(E,F):=\mathcal{L}(E,F)\cap\mathcal{I}$ satisfy: 
%\begin{description}
(i) $\mathcal{I}(E,F)$ is a vector subspace of $\mathcal{L}(E,F)$ and contains all finite-rank operators $\mathcal{F}(E,F)$; 
(ii) (\emph{Ideal property}) If $v\in\mathcal{L}(F_0,F)$, $T\in\mathcal{I}(E_0,F_0)$, and $u\in\mathcal{L}(E,E_0)$, then $v\circ T\circ u\in\mathcal{I}(E,F)$.
%\end{description}
We say that $\mathcal{I}$ is a \emph{quasi-normed operator ideal} if there exists a map $\alpha:\mathcal{I}\to[0,\infty)$ such that the restriction of $\alpha$ to each component $\mathcal{I}(E,F)$ is a quasi-norm for all Banach spaces $E$ and $F$; $\alpha(\mathrm{id}_{\mathbb{K}})=1$; and for all $T_1\in\mathcal{L}(F_0,F)$, $T_2\in\mathcal{I}(E_0,F_0)$, and $T_3\in\mathcal{L}(E,E_0)$,
\(
\alpha(T_1\circ T_2\circ T_3)\le\|T_1\|\cdot\alpha(T_2)\cdot\|T_3\|.
\)
Let $(\mathcal{I},\alpha)$ be a quasi-normed operator ideal. If $\mathcal{I}$ is complete with respect to $\alpha$, we call it a \emph{quasi-Banach operator ideal}.

Let $\mathcal{S}$ be a standard sequence space. The next result is a particular case of \cite[Proposition 2.2]{Albuquerque-Coleta} for operator ideals. For simplicity of notation, given a map $T:W \to Z$ we denote by $Tw$ the evaluation of $T$ at $w$; for a sequence $x\in E^{\mathbb{N}}$, we write $x(j) \in E$ for its $j$-th coordinate.

\begin{proposition}\label{prop_pw-mothervector}
Let $\mathcal{I}$ be a (quasi-)Banach operator ideal. Given a partition $\left( \mathbb{N}_k \right)_{k \in \mathbb{N}}$ of $\mathbb{N}$ and a non-trivial operator $T \in \mathcal{I} \left(E;\mathcal{S}\right)$, we have the following.
\begin{enumerate}[(i)]
\item For each natural $k \geq 1$, the map $T_k : E \to  \mathcal{S}$ defined by
\[
T_k x = \sum_{i \in \mathbb{N}} Tx(i) \cdot e_{n_i^{(k)}}.
\]
%\begin{equation*}
%T_k x \left( j \right) := 
%\begin{cases}
%0, & \mbox{ if } j \notin \mathbb{N}_k,\\
%Tx (i), & \mbox{ if } j = n_i^{(k)} \in \mathbb{N}_k,
%\end{cases} 
%\end{equation*}
lies in $\mathcal{I} \left(E;\mathcal{S}\right)$, and $\left\| T_{k} x \right\|_{\mathcal{S}} =\left\| Tx \right\|_{\mathcal{S}}$ for all $x\in E$. Moreover, $\left\{T_k : k \geq 2 \right\} \cup \{T\}$ is linearly independent.

\item For some $0<p\leq1$, the map $\Psi :\ell_p \to  \mathcal{I} \left(E;\mathcal{S}\right)$ given by
\begin{equation*} %\label{psi_op}
\Psi_a := \Psi a = a_1 T + \sum_{j=2}^{\infty} a_j T_j,  \quad  a=(a_j)_{j\in \mathbb{N}} \in \ell_p,
\end{equation*}
is well-defined, linear, bounded, and injective.
%Moreover, $\Psi\left( \ell_p \right) \supset \left\{T_k : k \in \mathbb{N}\right\} \cup \{T\}$.
\end{enumerate}
\end{proposition}

Recently, Botelho and Campos provided a characterization of summing operators via transformations of vector-valued sequence classes (see \cite{Botelho-Campos}). We reformulate this notion in the framework of standard sequence classes: given standard sequence classes \( X \) and \( Y \), we say that a linear operator \( u \in \mathcal{L}(E;F) \) is \emph{(X;Y)-summing} if \((u(x_j))_{j=1}^\infty \in Y(F)\) whenever \((x_j)_{j=1}^\infty \in X(E)\). In this case, we write \( u \in \mathcal{L}_{X;Y}(E;F) \).

Some examples of $(X;Y)$-summing operators are the ideals $\Pi_{p,q}(E;F)$ of absolutely $(p,q)$-summing operators ($0<q\leq p< \infty$) (see \cite[Chapter 10.]{Diestel}) and  $\mathcal{D}_{p,q}(E;F)$ of Cohen strongly $(p,q)$-summing operators ($1< q \le p < \infty$) (see \cite[Section 3.]{Apiola}). In these cases, we have
\begin{equation}\label{pip}
\Pi_{p,q}(E;F)
= \mathcal{L}_{\ell_q^w; \ell_p(\cdot)}(E;F)
= \mathcal{L}_{\ell_q^u; \ell_p(\cdot)}(E;F)\ \ \textrm{and}\ \  \mathcal{D}_{p,q}(E;F)
= \mathcal{L}_{\ell_q(\cdot);\ell_p\langle \cdot\rangle}(E;F).
\end{equation}

In the same direction, we employ the notion of linear stability for sequence classes, introduced by Botelho and Campos (see \cite[Definition~3.2]{Botelho-Campos}) in the context of standard sequence classes. A standard sequence class \(X\) is called \emph{linearly stable} if \(\mathcal{L}_{X;X} = \mathcal{L}\), meaning that for all Banach spaces \(E\) and \(F\) and every operator \(T \in \mathcal{L}(E;F)\), it follows that \((T(x_j))_{j=1}^\infty \in X(F)\) whenever \((x_j)_{j=1}^\infty \in X(E)\). All classes presented in Example~\ref{ex-seqclass} are linearly stable. The following properties, which are expected in this setting, will be useful later.

\begin{proposition}
Let $\mathcal{S}$ be a standard Banach sequence space. If $X$ is a linearly stable standard sequence class, then the following conditions hold:
\begin{enumerate}[(i)]
%(\emph{Subsequence stability})
\item If $(x^{(i)})_{i=1}^\infty \in X(\mathcal{S})$ and $(x^{(i)}_{n_k})_{k\in\mathbb{N}}$ is a subsequence of $x^{(i)}$ for each $i$, then $((x^{(i)}_{n_k})_{k\in\mathbb{N}})_{i\in\mathbb{N}} \in X(\mathcal{S})$ and
\[
\big\|((x^{(i)}_{n_k})_{k\in\mathbb{N}})_{i\in\mathbb{N}}\big\|_{X(\mathcal{S})}
\leq 
\|(x^{(i)})_{i=1}^\infty\|_{X(\mathcal{S})}.
\]

%(\emph{Zero-spreading stability})
\item If $(x^{(i)})_{i=1}^\infty \in X(\mathcal{S})$ and $\mathbb{N}'=\{n_1<n_2<n_3<\cdots\}$ is an infinite subset of $\mathbb{N}$, then for all $i \in \mathbb{N}$ the $\mathcal{S}$-valued sequence $(y^{(i)})_{i=1}^\infty$ defined by
\[
y^{(i)}=\sum_{k\in\mathbb{N}} x^{(i)}_k\ \cdot e_{n_k}
\]
belongs to $X(\mathcal{S})$ and satisfies $
\|(y^{(i)})_{i=1}^\infty\|_{X(\mathcal{S})}
\leq
\|(x^{(i)})_{i=1}^\infty\|_{X(\mathcal{S})}$.
\end{enumerate}
\end{proposition}
\begin{proof}
Since $\mathcal{S}$ is a standard sequence space, items (ii) and (iii) of
Definition \ref{standard-sc} ensure that, for any infinite subset
$\mathbb{N}'=\{n_1<n_2<n_3<\cdots\}\subset \mathbb{N}$, we may consider the
linear and continuous operators
$\operatorname{R}_{\mathbb{N}'}, \operatorname{SP}_{\mathbb{N}'}
\in \mathcal{L}(\mathcal{S};\mathcal{S})$ defined by
\[
\operatorname{R}_{\mathbb{N}'}((x_j)_{j\in\mathbb{N}})=(x_j)_{j\in\mathbb{N}'}
\quad\text{and}\quad
\operatorname{SP}_{\mathbb{N}'}((x_j)_{j\in\mathbb{N}})
= \sum_{i\in\mathbb{N}} x_i\cdot e_{n_i},
\]
for all $(x_j)_{j\in\mathbb{N}}\in\mathcal{S}$. Therefore, the result follows
from the linear stability of $X$.
\end{proof}

If $X$ is a linearly stable class, we obtain, by a similar argument, a result analogous to Proposition~\ref{lema-restricoes}. The proof follows the same reasoning and is omitted.

\begin{proposition} \label{lema-restricoes-2}
Let $X$ be a linearly stable standard sequence class and let $\mathcal{S}$ be a standard sequence space. 
Let $(x^{(i)})_{i=1}^\infty$ be in ${\mathcal{S}}^{\mathbb{N}}$, where each $x^{(i)} = (x^{(i)}_j)_{j=1}^\infty$ belongs to $\mathcal{S}$, and assume that $(x^{(i)})_{i=1}^\infty \notin X(\mathcal{S})$.
Then, for every partition $\mathbb{N} = \mathbb{N}' \cup \mathbb{N}''$,
\[
(x^{(i)}|_{\mathbb{N}'})_{i \in \mathbb{N}} \notin X(\mathcal{S})
\quad\text{or}\quad
(x^{(i)}|_{\mathbb{N}''})_{i \in \mathbb{N}} \notin X(\mathcal{S}).
\]
\end{proposition}

\begin{theorem}\label{th-opideals}
Let $\mathcal{S}$ be a Banach standard sequence space, and let $X$ and $Y$ be standard sequence classes, with $Y$ linearly stable. Given a Banach space $E$ and a (quasi-)Banach operator ideal $\mathcal{I}(E;\mathcal{S})$, the set
$
\mathcal{I}(E;\mathcal{S}) \setminus \mathcal{L}_{X;Y}(E;\mathcal{S})
$
is either empty or pointwise $\mathfrak{c}$-spaceable.
\end{theorem}
\begin{proof}
Assume that there exists a non-trivial operator $T \in \mathcal{I}(E;\mathcal{S}) \setminus \mathcal{L}_{X;Y}(E;\mathcal{S})$. Fix $x_0 \in E$ and $j_0 \in \mathbb{N}$ such that $Tx_0(j_0) \neq 0$. Choose $z = (z_n)_{n=1}^\infty \in X(E)$ such that $(Tz_n)_{n=1}^\infty \notin Y(\mathcal{S})$. By Proposition~\ref{lema-restricoes-2}, there exists an infinite subset 
$\mathbb{N}_1 \subset \mathbb{N}$ with infinite complement and with $j_0 \in \mathbb{N}_1$ such that
\begin{equation}\label{z-dem}
\big( Tz_n|_{\mathbb{N}_1} \big)_{n \in \mathbb{N}} \notin Y(\mathcal{S}).
\end{equation}

Let $\mathbb{N}\setminus \mathbb{N}_1 = \bigcup_{k \ge 2} \mathbb{N}_k$ be a partition into infinite subsets. By Proposition~\ref{prop_pw-mothervector}, we obtain a linearly independent sequence of operators 
$(T_k)_{k \geq 2} \subset \mathcal{I}(E;\mathcal{S})$. We claim $T_k \notin \mathcal{L}_{X;Y}(E;\mathcal{S}), \, k \geq 2$. Indeed, restricting $T_kz_n$ into the coordinates of $\mathbb{N}_k$ yields
\[
(Tz_n)_{n\in\mathbb{N}}
= \big( T_kz_n|_{\mathbb{N}_k} \big)_{n\in\mathbb{N}}
\]
which cannot belong to $Y(\mathcal{S})$. Moreover, Proposition~\ref{prop_pw-mothervector} yields, for some $0<p\le 1$, a well-defined, bounded, linear, injective map 
$\Psi : \ell_p \to \mathcal{I}(E;\mathcal{S})$. Hence $\Psi(\ell_p)$ is a $\mathfrak{c}$-dimensional subspace of 
$\mathcal{I}(E;\mathcal{S})$ containing $T$.

\vskip 3mm
\noindent\textbf{Claim 1.} 
$\Psi(\ell_p)\setminus\{0\} \subset 
\mathcal{I}(E;\mathcal{S}) \setminus \mathcal{L}_{X;Y}(E;\mathcal{S})$.

\vskip 2mm
Let $a=(a_j)_{j=1}^\infty \in \ell_p \setminus\{0\}$. For all $n \in \mathbb{N}$,
\[
\Psi_a(z_n)
= a_1 Tz_n + \sum_{j=2}^\infty a_j T_jz_n.
\]
First suppose that $a_1 \neq 0$. Since each $T_jz_n$ has zero coordinates outside $\mathbb{N}_j$ for $j \ge 2$, we have
\(
\Psi_az_n|_{\mathbb{N}_1}
= a_1 \, Tz_n|_{\mathbb{N}_1}.
\)
Thus
\[
\big( \Psi_az_n|_{\mathbb{N}_1} \big)_{n\in\mathbb{N}}
= \big( a_1 Tz_n|_{\mathbb{N}_1} \big)_{n\in\mathbb{N}}
\notin Y(\mathcal{S}),
\]
implies $\Psi a$ does not belong to $\mathcal{L}_{X;Y}(E;\mathcal{S})$.

Now assume that $a_1 = 0$, and choose $i$ such that $a_i \neq 0$. Since $T_iz_n$ vanishes outside $\mathbb{N}_i$,
\[
(a_i Tz_n)_{n\in\mathbb{N}}
= \big( a_i T_iz_n|_{\mathbb{N}_i} \big)_{n\in\mathbb{N}}
= \big( \Psi_az_n|_{\mathbb{N}_i} \big)_{n\in\mathbb{N}}
\notin Y(\mathcal{S}),
\]
which also yields $\Psi a \notin \mathcal{L}_{X;Y}(E;\mathcal{S})$. Thus the claim is proved, establishing pointwise $\mathfrak{c}$-lineability. We now prove the spaceability statement of the main result.

\vskip 3mm

\noindent\textbf{Claim 2.}  $\overline{\Psi(\ell_p)}\setminus\{0\} \subset 
\mathcal{I}(E;\mathcal{S}) \setminus \mathcal{L}_{X;Y}(E;\mathcal{S})$.

\vskip 2mm
Let us fix a non-zero operator $S = \lim_{k}\Psi_{a^{(k)}} \in \overline{\Psi(\ell_p)}$, where each $a^{(k)} = (a^{(k)}_j)_{j \in \mathbb{N}} \in \ell_p$. For every $b = (b_j)^\infty_{j=1} \in \ell_p$, $x \in E$ and $j \in \mathbb{N}_1$, we have $\Psi_b x (j) = b_1 T x (j)$. Choose $x_0 \in E$ and $j_0 \in \mathbb{N}_1$ such that $Tx_0 (j_0) \neq 0$. Then
\[
Sx_0(j_0) = \lim_{k}\Psi_{a^{(k)}}x_0(j_0) = \left( \lim_{k}a_1^{(k)}\right)Tx_0(j_0)
\]
so the limit \(\lim_k a^{(k)}_1\) exists. First, suppose that $\lim_{k}a_1^{(k)} \neq 0$. Since
\[
Sz_n=\lim_{k}\left[a^{(k)}_1\,Tz_n+\sum_{j=2}^{\infty} a^{(k)}_j\,T_jz_n\right] 
\qquad \forall\, n\in\mathbb N,
\]
if \((Sz_n)_{n\in\mathbb N}\in Y(\mathcal S)\), then
\[
\Big(\big(\lim_{k} a^{(k)}_1\big)\, Tz_n|_{\mathbb N_1}\Big)_{n\in\mathbb N}
=
\big(Sz_n|_{\mathbb N_1}\big)_{n\in\mathbb N}
\in Y(\mathcal S),
\]
which contradicts \eqref{z-dem}.

Now consider the case $\lim_{k}a_1^{(k)} = 0$. Since $S \neq 0$, there exist $x \in E$ and $i_0 \in \mathbb{N}$ such that $Sx(i_0)\neq 0$. Choose $r,t \in \mathbb{N}$ such that $i_0 = n_t^{(r)}$. Then
\[
Sx(i_0) =
\lim_{k}\left[ \sum_{j = 1}^\infty a_j^{(k)}T_jx(i_0)\right] =
\left( \lim_{k}a_r^{(k)}\right) T_rx(n_t^{(r)})=
\left( \lim_{k}a_r^{(k)}\right) Tx(t)
\]
so $\lim_{k}a_r^{(k)} \neq 0$. If $(Sz_n)_{n \in \mathbb{N}}\in Y(\mathcal{S})$, then
\[
\left(\lim_{k}a_r^{(k)}Tz_n \right)_{n \in \mathbb{N}} =
\left(\lim_{k}\sum_{j = 1}^\infty a_j^{(k)}T_jz_n|_{\mathbb{N}_r} \right)_{n \in \mathbb{N}}=
(Sz_n|_{\mathbb{N}_r})_{n \in \mathbb{N}}\in Y(\mathcal{S})
\]
and hence $(Tz_n)_{n \in \mathbb{N}}\in Y(\mathcal{S})$, which cannot occur. Thus this proves the claim and, therefore, the pointwise $\mathfrak{c}$-spaceability.
\end{proof}

\section{Applications}\label{SEC-APLIC}

We present several applications of our primary results. First, we discuss the geometric spaceability of standard sequence classes (Theorem \ref{th-standardseq}), as introduced in Section \ref{Sec-LIN-SSS}. Second, we address pointwise spaceability within operator ideals (Theorem \ref{th-opideals}), established in Section \ref{Sec-PL-OI}. These applications not only provide generalizations of existing results in the literature but also establish several new developments in the field.

\subsection{Applications in standard sequence spaces}

In what follows, we provide immediate applications of Theorem \ref{th-standardseq}, adopting the notation established in Section \ref{Sec-LIN-SSS}. We recall the identity
\[
X(E) \setminus \bigcup_{\lambda \in \Lambda} Y_\lambda(E) = G(X, Id_E, \{Y_\lambda\}_{\lambda \in \Lambda}).
\]
Furthermore, in the second corollary, we make use of the properties established in Example \ref{ex-detach}~b).

\begin{corollary} \label{cor1}
Let $X$ be a standard sequence class, $\{Y_\lambda\}_{\lambda \in \Lambda}$ be an $E$-nested family of standard sequence classes such that $(X,Id_E,\{Y_\lambda\}_{\lambda \in \Lambda})$ is detachable. If $X(E) \setminus\bigcup_{\lambda \in \Lambda}Y_\lambda(E)$ is a non-empty set, then it is $(\alpha, \mathfrak{c})$-spaceable if and only if $\alpha < \aleph_0$.
\end{corollary}

\begin{corollary}
Let $M$ be an Orlicz function and $X$ be a standard sequence class such that $X(E) \subseteq c_0(E)$. If $X(E) \setminus\ell_M(E)$ is a non-empty set, then it is $(\alpha, \mathfrak{c})$-spaceable if and only if $\alpha < \aleph_0$.
\end{corollary}

Let us look at some particular instances obtained by the above corollaries. These results generalize (and recover) those present in \cite{Araujo-Barbosa-Raposo-Ribeiro,Barroso-Botelho-Favaro-Pellegrino} and \cite{Botelho-Cariello-Favaro-Pellegrino}. The non-emptiness of the differences in examples a) to d) below are given in \cite{Botelho-Cariello-Favaro-Pellegrino} where the maximal spaceability of these sets is proven.

\begin{example}\rm For all sets presented here and for all Banach space $E$ we have $(\alpha, \mathfrak{c})$-spaceability if and only if $\alpha < \aleph_0$.

\vskip 2mm

\noindent a) The sets
$$
c_0(E)\setminus \bigcup_{0<p<\infty}\ell_p^w(E)
\quad \textrm{and} \quad c_0(E)\setminus \bigcup_{0<p<\infty}\ell_p(E).
$$

\vskip 2mm

\noindent b) For $0<p<\infty$, the sets
\[
\ell_p(E)\setminus \bigcup_{0<q<p}\ell_q^w(E)
\quad \text{and} \quad
\ell_p(E)\setminus \bigcup_{0<q<p}\ell_q(E).
\]

\vskip 2mm

\noindent  c) For $0<p_1<p_2<\infty$ and $0 < q_1,q_2<\infty$, the set  $\ell_{p_2,q_2}(E)\setminus \ell_{p_1,q_1}(E)$.

\vskip 2mm

\noindent d) For $1\leq p<\infty$, the set 
$$
\ell_p\langle E\rangle \setminus \bigcup_{0<q<p}\ell_q^w(E).
$$

\vskip 2mm

\noindent e) For $1 < p<\infty$ and $X(\cdot) = \ell_p(\cdot)$ or $\ell_p\langle\cdot\rangle$, the set 
$$
X(E)\setminus \bigcup_{1\le q<p}\ell_q^{\rm mid}(E).
$$
Indeed, from b) and d) we have $X(E) \setminus \bigcup_{0<q<p}\ell_q^w(E) \neq \emptyset$ and $\ell_q^{\rm mid}(E) \subseteq \ell_q^w(E)$ for each $q \in [1,p)$, which give us $X(E)\setminus\bigcup_{1\le q<p}\ell_q^{\rm mid}(E) \neq \emptyset $.
\end{example}

If $u$ is a non-$(X;Y)$-summing operator, then the set $G(X,u,Y)$ is non-empty. As a consequence of this, we obtain the following corollary.

\begin{corollary}
Let  $u \in \mathcal{L}(E;F)$, $X$ and $Y$ be  standard sequence classes such that $X(E) \subset c_0(E)$ and $(X, u, Y)$ is detachable. If $u \notin \mathcal{L}_{X;Y}(E;F)$, then the set
\begin{equation*}
\{(x_j)^\infty_{j=1} \in X(E) : (u(x_j))^\infty_{j=1} \notin Y(F)\}
\end{equation*}
is $(\alpha,\mathfrak{c})$-spaceable if and only if $\alpha < \aleph_0$.
\end{corollary}

\begin{example}\rm
Using the characterization in \eqref{pip}, if $u \notin \mathcal{L}_{\ell_q^u;\ell_p(\cdot)}(E;F)$ ($0<q\leq p< \infty$) and if $v \notin \mathcal{L}_{\ell_q(\cdot);\ell_p\langle \cdot\rangle}(E;F)$ ($1< q \le p < \infty$), then the sets
$$
\{(x_j)^\infty_{j=1} \in \ell_q^u(E) : (u(x_j))^\infty_{j=1} \notin \ell_p(F)\}
$$
and
$$
\{(x_j)^\infty_{j=1} \in \ell_q(E) : (u(x_j))^\infty_{j=1} \notin \ell_p\langle F\rangle\}
$$
are $(\alpha,\mathfrak{c})$-spaceable if and only if $\alpha < \aleph_0$.
\end{example}

\subsection{Applications in operator ideals}

We explore immediate consequences of Theorem \ref{th-opideals} regarding the algebraic and topological structure of subsets of operator ideals. We begin by examining the pointwise spaceability of differences between classes of summing operators and then extend these results to more general operator ideals. We provide specific examples involving bounded operators (or completely continuous) that fail to be $(X;Y)$-summing to certain sequence classes.

\begin{corollary}
Let $p_1,q_1,p_2,q_2 \in (0,\infty)$ with $p_1 \le p_2$ and $q_1 \le q_2$, and let $\mathcal{S}$ be a Banach standard sequence space. For every Banach space $E$, the set
\[
\Pi_{p_2,q_2}(E;\mathcal{S}) \setminus \Pi_{p_1,q_1}(E;\mathcal{S})
\]
is either empty or pointwise $\mathfrak{c}$-spaceable.
\end{corollary}

\begin{corollary}
Let $\mathcal{S}$ be a Banach standard sequence space, and let $X$ and $Y$ be standard sequence classes, with $Y$ linearly stable. Given a Banach space $E$, the set
\[
\mathcal{L}(E;\mathcal{S}) \setminus \mathcal{L}_{X;Y}(E;\mathcal{S})
\]
is either empty or pointwise $\mathfrak{c}$-spaceable.
\end{corollary}

\begin{example}\rm 
(a) The set
\[
\mathcal{L}(E;\mathcal{S}) \setminus \mathcal{L}_{\ell_p^w;\ell_p^{\mathrm{mid}}}(E;\mathcal{S})
\]
is pointwise $\mathfrak{c}$-spaceable if and only if $\ell_p^{\mathrm{mid}}(E) \neq \ell_p^w(E)$ (see \cite[Theorem~2.6]{Botelho-Campos-Santos}).

\vskip 2mm

\noindent (b) The set
\[
\mathcal{L}(E;\mathcal{S}) \setminus \mathcal{L}_{\ell_p^{\mathrm{mid}};\ell_p(\cdot)}(E;\mathcal{S})
\]
is pointwise $\mathfrak{c}$-spaceable if and only if $\mathcal{S}$ is not a subspace of $L_p(\mu)$ for any Borel measure $\mu$ (see \cite[Theorem~2.7]{Botelho-Campos-Santos}).

\vskip 2mm

\noindent (c) For $0<q_1,q_2,r<\infty$, the set
\[
\mathcal{L}(E;\mathcal{S}) \setminus \mathcal{L}_{\ell_r^w;\ell_{p_{1},q_1}(\cdot)}(E;\mathcal{S})
\]
is either empty or pointwise $\mathfrak{c}$-spaceable.

\vskip 2mm

\noindent (d) For $1 \le p < \infty$, the set
\[
\mathcal{L}(E;\mathcal{S}) \setminus \mathcal{L}_{\ell_p(\cdot);\ell_p\langle \cdot \rangle}(E;\mathcal{S})
\]
is either empty or pointwise $\mathfrak{c}$-spaceable.
\end{example}

As in the classical theory of operator ideals, we consider $\mathcal{C}\mathcal{C}$ to be the ideal of completely continuous operators (see \cite[Section~1.6]{Pietsch2}). More precisely, completely continuous operators send weakly convergent sequences into norm-convergent sequences.

\begin{corollary}
Let $\mathcal{S}$ be a Banach standard sequence space, and let $X$ and $Y$ be standard sequence classes, with $Y$ linearly stable. Given a Banach space $E$, the set
\[
\mathcal{CC}(E;\mathcal{S}) \setminus \mathcal{L}_{X;Y}(E;\mathcal{S})
\]
is either empty or pointwise $\mathfrak{c}$-spaceable.
\end{corollary}

\begin{example}\rm
The set $\mathcal{C}\mathcal{C}(\ell_1;\ell_1) \setminus \Pi_p(\ell_1;\ell_1)$ is pointwise $\mathfrak{c}$-spaceable for every $0<p<\infty$. Indeed, since $\ell_1$ has the Schur property, the identity operator $\operatorname{Id}_{\ell_1}:\ell_1 \to \ell_1$ is completely continuous. On the other hand, $\operatorname{Id}_{\ell_1}$ never is absolutely $p$-summing by the Dvoretzky-Rogers theorem (see \cite[Theorem 2.18]{Diestel}).
\end{example}

\section*{Acknowledgments and Funding}

\noindent Nacib Gurgel Albuquerque was partially supported in part by CNPq Grants 406457/2023-9 and 403964/2024-5. Luiz Felipe de Pinho Sousa was supported by a CAPES scholarship and CNPq Grant 406457/2023-9.

\end{document}